\documentclass[a4paper,oneside,10pt]{article}

\usepackage[english]{babel}
\usepackage[latin1]{inputenc}
\usepackage[T1]{fontenc}
\usepackage{lmodern}
\usepackage{cite}
\usepackage{microtype}
\usepackage{tikz}
\usetikzlibrary{patterns}
\usepackage{perpage}
\MakePerPage{footnote}

\usepackage{amssymb}
\usepackage{amsmath}
\usepackage{mathrsfs}
\usepackage{dsfont}
\usepackage{exscale}
\usepackage{relsize}
\usepackage{amsthm}

\newcommand{\ind}[1]{\mathds{1}_{#1}}

\newcommand{\R}{\mathbb{R}}

\newcommand{\E}{\mathbb{E}}

\newcommand{\Ac}{\mathcal{A}}
\newcommand{\Bc}{\mathcal{B}}

\newcommand{\Hc}{\mathcal{H}}

\newcommand{\Nc}{\mathcal{N}}

\newcommand{\Sc}{\mathcal{S}}

\renewcommand{\d}{\delta}
\newcommand{\D}{\Delta}
\renewcommand{\epsilon}{\varepsilon}
\newcommand{\eps}{\varepsilon}
\newcommand{\z}{\zeta}
\renewcommand{\r}{\rho}
\newcommand{\s}{\sigma}

\renewcommand{\phi}{\varphi}

\renewcommand{\l}{\lambda}
\renewcommand{\L}{\Lambda}

\newcommand{\loi}{\mathscr{L}}

\def\Dro{\smash{{D}^{\!\!\!\!\raise4pt\hbox{$\scriptstyle o$}}}}
\def\Bro{\smash{{B}^{\!\!\!\!\raise4pt\hbox{$\scriptstyle o$}}}}
\def\Ccro{\smash{{\mathcal{C}}^{\!\!\!\raise4pt\hbox{$\scriptstyle o$}}}}
\def\Aro{\smash{{A}^{\!\!\!\raise5pt\hbox{$\scriptstyle o$}}}}
\def\Aro2{\smash{{A}^{\!\!\!\raise4pt\hbox{$\scriptstyle o$}}}}

\newcommand{\limsupn}{\underset{n \to +\infty}{\mathrm{limsup}}\,}
\newcommand{\liminfn}{\underset{n \to +\infty}{\mathrm{liminf}}\,}

\newtheorem{theo}{Theorem}

\newtheorem{prop}[theo]{Proposition}

\newtheorem{lem}[theo]{Lemma}

\renewenvironment{proof}{\noindent{\bf Proof.}}{\qed}

\usepackage{tocloft}
\begin{document}

\renewcommand{\contentsname}{Contents}
\renewcommand{\refname}{\textbf{References}}
\renewcommand{\abstractname}{Abstract}

\begin{center}
\begin{Huge}
The Curie-Weiss Model of SOC\medskip

in Higher Dimension
\end{Huge} \bigskip \bigskip \bigskip 

\begin{Large} Matthias Gorny \end{Large} \smallskip
 
\begin{large} {\it Universit\'e Paris-Sud and ENS Paris} \end{large} \bigskip \bigskip \bigskip 
\end{center}

\begin{abstract}
\noindent We build and study a multidimensional version of the Curie-Weiss model of self-organized criticality we have designed in~\cite{CerfGorny}. For symmetric distributions satisfying some integrability condition, we prove that the sum~$S_n$ of the randoms vectors in the model has a typical critical behaviour. The fluctuations are of order $n^{3/4}$ and the limiting law has a density proportional to the exponential of a fourth-degree polynomial.
\end{abstract}
\bigskip \bigskip



\bigskip \bigskip \bigskip

\section{Introduction}

\noindent In~\cite{CerfGorny} and~\cite{GORGaussCase}, we introduced a \textit{Curie-Weiss model of self-organized criticality} (SOC): we transformed the distribution associated to the generalized Ising Curie-Weiss model by implementing an automatic control of the inverse temperature which forces the model to evolve towards a critical state. It is the model given by an infinite triangular array of real-valued random variables $(X_{n}^{k})_{1\leq k \leq n}$ such that, for all $n \geq 1$, $(X^{1}_{n},\dots,X^{n}_{n})$ has the distribution
\[\frac{1}{Z_{n}}\exp\left(\frac{1}{2}\frac{(x_{1}+\dots+x_{n})^{2}}{x_{1}^{2}+\dots+x_{n}^{2}}\right)\ind{\{x_{1}^{2}+\dots+x_{n}^{2}>0\}}\,\prod_{i=1}^{n}d\r(x_{i}),\]
where $\r$ is a probability measure on $\R$ which is not the Dirac mass at~$0$, and where $Z_n$ is the normalization constant. We extended the study of this model in~\cite{DynCWSOC},~\cite{Gorny3},~\cite{GornyThesis} and~\cite{GorVar}. For symmetric distributions satisfying some exponential moments condition, we proved that the sum $S_{n}$ of the random variables behaves as in the typical critical generalized Ising Curie-Weiss model: the fluctuations are of order $n^{3/4}$ and the limiting law is $C \exp(-\lambda x^{4})\,dx$ where $C$ and $\lambda$ are suitable positive constants. Moreover, by construction, the model does not depend on any external parameter. That is why we can conclude it exhibits the phenomenon of self-organized criticality (SOC). Our motivations for studying such a model are detailed in~\cite{CerfGorny}. \medskip

\noindent Let $d\geq 1$. In this paper we define a $d$-dimensional version of the Curie-Weiss model of SOC, i.e, such that the $X^{k}_{n}$, $1\leq k\leq n$, are random vectors in $\R^d$. Let us start by defining the \textit{$d$-dimensional generalized Ising Curie-Weiss model}. Let~$\r$ be a symmetric probability measure on $\R^d$ such that
\[\forall v \geq 0 \qquad \int_{\R^d}\exp(v\|z\|^{2})\,d\r(z)<\infty\,.\]
Assume that its covariance matrix
\[\Sigma=\int_{\R^{d}}z\,^{t}\!z\,d\r(z)\]
is invertible (i.e, $\r$ is non-degenerate on $\R^d$). The $d$-dimensional generalized Ising Curie-Weiss model associated to~$\r$ and to the temperature field $T$ (which is here a $d\times d$ symmetric positive definite matrix) is defined through an infinite triangular array of random vectors $(X_{n}^{k})_{1\leq k \leq n}$ such that, for all $n \geq 1$, $(X^{1}_{n},\dots,X^{n}_{n})$ has the distribution
\[\frac{1}{Z_{n}(T)} \exp\left(\frac{1}{2n} \big\langle T^{-1}(x_1+\dots+x_n),(x_1+\dots+x_n)\big\rangle \right) \prod_{i=1}^{n} d\r(x_{i})\,,\]
where $Z_{n}(T)$ is a normalization. When $d=1$ and $\r=(\d_{-1}+\d_{1})/2$, we recover the classical Ising Curie-Weiss model. Let $S_n=X_{n}^{1}+\dots+X_{n}^{n}$ for any $n\geq 1$. By extending the methods of Ellis and Newmann (see~\cite{EN}) to the higher dimension, we obtain that, under some~\guillemotleft~sub-Gaussian~\guillemotright~hypothesis on $\r$, if $T-\Sigma$ is a symmetric positive definite matrix, then
\[\frac{S_{n}}{\sqrt{n}} \overset{\loi}{\underset{n \to +\infty}{\longrightarrow}}\Nc_{d}\big(0,T(T-\Sigma)^{-1}\Sigma\big),\]
the centered $d$-dimensional Gaussian distribution with covariance matrix\linebreak $T(T-\Sigma)^{-1}\Sigma$. If $T=\Sigma$ (critical case) then
\[\frac{S_{n}}{n^{3/4}} \overset{\loi}{\underset{n \to +\infty}{\longrightarrow}}C_{\r} \exp\left(-\phi_{\r}(s_1,\dots,s_d)\right)\,ds_1\,\cdots\,ds_d,\]
where $C_{\r}$ is a normalization constant and $\phi_{\r}$ is an homogeneous polynomial of degree four in $\R[X_1,\dots,X_d]$ such that $\exp(-\phi_{\r})$ is integrable with respect to the Lebesgue measure on $\R^d$. Detailed proofs of these results are given in section~23 of~\cite{GornyThesis}. These results highlight that the non-critical fluctuations are normal (in the Gaussian sense) while the critical fluctuations are of order $n^{3/4}$ (or eventually $n^{1-1/2k}$, $k\geq 3$, in the degenerate cases of sub-Gaussian measures, see~\cite{EN}). \medskip

\noindent Now we try to modify this model in order to construct a $d$-dimensional SOC model. As in~\cite{CerfGorny}, we search an automatic control of the temperature field $T$, which would be a function of the random variables in the model, so that, when $n$ goes to $+ \infty$, $T$ converges towards the critical value $\Sigma$ of the model. We start with the following observation: if $(Y_{n})_{n \geq 1}$ is a sequence of independent random vectors with identical distribution~$\r$, then, by the law of large numbers,
\[\frac{\widehat{\Sigma}_{n}}{n}\overset{\mbox{a.s}}{\underset{n \to +\infty}{\longrightarrow}} \Sigma,\]
where
\[\forall n \geq 1 \qquad \widehat{\Sigma}_{n}=X^{1}_{n}\,^{t}\!(X^{1}_{n})+\dots+X^{n}_{n}\,^{t}\!(X^{n}_{n}).\]
This convergence provides us with an estimator of $\Sigma$. If we believe that a similar convergence holds in the $d$-dimensional generalized Ising Curie-Weiss model, then we are tempted to~\guillemotleft~replace $T$ by $\widehat{\Sigma}_{n}/n$~\guillemotright~in the previous distribution. Hence, in this paper, we consider the following model:\medskip

\noindent \textbf{The model.} Let $(X_{n}^{k})_{n \geq d,\,1\leq k \leq n}$ be an infinite triangular array of random vectors in $\R^{d}$ such that, for any $n \geq d$, $(X^{1}_{n},\dots,X^{n}_{n})$ has the distribution $\smash{\widetilde{\mu}_{n,\r}}$, the probability measure on $(\R^{d})^{n}$ with density
\begin{align*}
(x_{1},\dots,x_{n})\longmapsto\frac{1}{Z_{n}} \exp\left(\frac{1}{2} \left\langle \left(\sum_{i=1}^{n}x_{i}\,^{t}\!x_{i}\right)^{-1}\left(\sum_{i=1}^{n}x_{i}\right),\left(\sum_{i=1}^{n}x_{i}\right) \right\rangle\right)
\end{align*}
with respect to $\r^{\otimes n}$ on the set
\[D_{n}^{+}=\left\{\, (x_{1},\dots,x_{n}) \in (\R^{d})^{n}: \mathrm{det}\left(\sum_{i=1}^{n}x_{i}\,^{t}\!x_{i}\right)>0\,\right\},\]
where
\[Z_{n}=\int_{D_{n}^{+}}\exp\left(\frac{1}{2} \left\langle \left(\sum_{i=1}^{n}x_{i}\,^{t}\!x_{i}\right)^{-1}\left(\sum_{i=1}^{n}x_{i}\right),\left(\sum_{i=1}^{n}x_{i}\right) \right\rangle\right)\prod_{i=1}^{n}\,d\r(x_{i}).\]
For any $n \geq d$, we denote $S_{n}=X^{1}_{n}+\dots+X^{n}_{n} \in \R^{d}$ and
\[T_{n}=X^{1}_{n}\,^{t}\!(X^{1}_{n})+\dots+X^{n}_{n}\,^{t}\!(X^{n}_{n}).\]

\noindent In section~\ref{Preliminaries}.\ref{WellDefined}, we prove rigorously that this model is well-defined, i.e.,\linebreak  $Z_n \in \,]0,+\infty[$ for any $n\geq d$.\medskip

\noindent According to the construction of this model and according to our results in one dimension, we expect that the fluctuations are of order $n^{3/4}$. Our main theorem states that they are indeed:

\begin{theo} Let $\r$ be a symmetric probability measure on $\R^d$. Suppose that
\[\exists v_0> 0 \qquad \int_{\R^{d}}\exp\left(v_0\|z\|^2\right)\,d\r(z)<\infty
\tag{$*$}\]
and that the $\r$-measure of any vector hyperplane of $\R^d$ is less than $1/\sqrt{e}$.\linebreak Let $\Sigma$ be the covariance matrix of $\r$ and let $M_4$ be the function defined on~$\R^d$~by
\[\forall z\in \R^d\qquad M_4(z)=\sum_{1\leq i,j,k,l\leq d}\left(\int_{\R^d}y_iy_jy_ky_l\,d\r(y)\right)z_iz_jz_kz_l.\]

\noindent \emph{Law of large numbers:} Under $\smash{\widetilde{\mu}_{n,\r}}$, $(S_n/n,T_n/n)$ converges in probability to~$(0,\Sigma)$.\medskip

\noindent \emph{Fluctuation result:} Under $\smash{\widetilde{\mu}_{n,\r}}$,
\[\frac{S_n}{n^{3/4}}\overset{\loi}{\underset{n \to \infty}{\longrightarrow}}
\frac{\displaystyle{\exp\left(-\frac{1}{12}M_4\big(\Sigma^{-1}z\big)\right)\,dz}}{\displaystyle{\int_{\R^d}\exp\left(-\frac{1}{12}M_4\big(\Sigma^{-1}u\big)\right)\,du}}.\]
\label{CVloiDimd}
\end{theo}

\noindent We prove that the matrix $\Sigma$ is invertible in subsection~\ref{Preliminaries}.\ref{SigmaInv}. After giving large deviations results in subsection~\ref{Preliminaries}.\ref{LargeDev}, we show the law of large numbers in section~\ref{ConvProb}. Finally, in section~\ref{AstuceCasGenDimd}, we prove that the function
\[z\longmapsto \exp\left(-M_4\left(\Sigma^{-1/2}z\right)/12\right)\]
is integrable on $\R^d$ and that $S_n/n^{3/4}$ converges in distribution to the announced limiting distribution.\medskip

\noindent Remark : in the case where $d=1$, we have already proved this theorem in~\cite{CerfGorny},~\cite{Gorny3} and~\cite{GorVar}. Moreover we succeeded to remove the hypothesis on the $\r$-measure of the vector hyperplanes (which turns out to be simply $\r(\{0\})<1/\sqrt{e}$ when $d=1$) with a conditioning argument. It seems not immediate that such arguments could extend in the case where~$d\geq 2$. However this assumption together with condition $(*)$ are technical hypothesis and we believe that the result should be true if $\r$ is only a non-degenerate symmetric probability measure on $\R^d$ having a finite fourth moment.

\section{Preliminaries}
\label{Preliminaries}

\noindent In this section, we suppose that $\r$ is a symmetric probability measure on $\R^d$ satisfying $(*)$ and such that the $\r$-measure of any vector hyperplane of $\R^d$ is less than $1/\sqrt{e}$.

\subsection{$\Sigma$ is a symmetric positive definite matrix}
\label{SigmaInv}

\noindent Since $\r$ satisfies condition $(*)$, the covariance matrix~$\Sigma$ is well-defined. It is of course a symmetric positive semi-definite matrix. Let $\Hc$ be an hyperplane of $\R^d$. If $\Hc$ is a vector hyperplane then, by hypothesis, $\r(\Hc)<1/\sqrt{e}<1$. If $\Hc$ is an affine (but not vector) hyperplane then,
\[\r(\Hc)=\r(-\Hc)=\frac{1}{2}(\r(\Hc)+\r(-\Hc))\leq \frac{1}{2}<1,\]
since $\r$ is symmetric and $\Hc\cap(-\Hc)=\varnothing$. In both cases $\r(\Hc)<1$ thus $\r$ is a non-degenerate probability measure on~$\R^d$. As a consequence $\Sigma$ is positive definite (see lemma~III.7 of~\cite{GornyThesis}).\medskip

\subsection{The model is well-defined}
\label{WellDefined}

\noindent Let us prove that the model is well defined, i.e., $Z_{n} \in \,]0,+\infty[$ for any $n\geq d$.

\begin{lem} Let $n\geq 1$ and let $x_{1},\dots,x_{n}$ be vectors in $\R^{d}$. We denote
\[A_{n}=x_{1}\,^{t}\!x_{1}+\dots+x_{n}\,^{t}\!x_{n}.\]
$\star$ If $n<d$, then $A_{n}$ is non-invertible.\\
$\star$ If $n=d$, then $A_{n}$ is invertible if and only if $(x_{1},\dots,x_{n})$ is a basis of $\R^{d}$.\\
$\star$ If $n>d$ and if the vectors $x_{1},\dots,x_{n}$ span $\R^d$, then $A_{n}$ is invertible.
\end{lem}

\begin{proof} $\star$ Let $n\leq d$. If $n<d$, we put $x_{n+1}=\dots=x_{d}=0$. We denote by $B$ the $d\times d$ matrix such that its columns are $x_{1},\dots,x_{d}$. We have then, for any $1\leq k,l \leq d$,
\[(B\,^{t}\!B)_{k,l}=\sum_{i=1}^{d}B_{k,i}B_{l,i}=\sum_{i=1}^{d} x_{i}(k)x_{i}(l)=\sum_{i=1}^{d} (x_{i}\,^{t}\!x_{i})_{k,l}=(A_{n})_{k,l}.\]
Therefore $A_{n}=B\,^{t}\!B$ and thus $A_{n}$ is invertible if and only if $B$ is invertible. As a consequence $A_n$ is invertible if and only if $(x_{1},\dots,x_{d})$ is a basis of $\R^{d}$. In the case where $n<d$, $B$ has at least a null column and thus is not invertible.
\smallskip

\noindent $\star$ Let $n>d$ and assume that the vectors $x_{1},\dots,x_{n}$ span $\R^{d}$. There exists then $1\leq i_1<\dots<i_d\leq n$ such that $(x_{i_1},\dots,x_{i_d})$ is a basis of $\R^d$. As a consequence, by the previous case, $A_{n}$ is the sum of a symmetric positive definite matrix and $n-d$ other symmetric positive semi-definite matrices. Therefore $A_{n}$ is definite thus invertible.
\end{proof}
\medskip

\noindent Let $n\geq d$. The non-degeneracy of $\r$ implies that its support is not included in a hyperplane of $\R^d$. As a consequence
\[\r^{\otimes n}\big(\{\,(x_1,\dots,x_n) \in (\R^d)^n : (x_1,\dots,x_d) \mbox{ is a basis of } \R^d\,\}\big)>0.\]
The previous lemma yields
\[\r^{\otimes n}\big(\{\,(x_1,\dots,x_n) \in (\R^d)^n : x_{1}\,^{t}\!x_{1}+\dots+x_{n}\,^{t}\!x_{n} \mbox{ is invertible}\,\}\big)>0,\]
i.e., $\smash{\r^{\otimes n}(D_{n}^{+})>0}$. Therefore $Z_n>0$.\medskip

\noindent We denote:\smallskip

\indent $\bullet$ $\Sc_{d}$ the space of $d\times d$ symmetric matrices.\smallskip

\indent $\bullet$ $\Sc_{d}^{+}$ the space of all matrices in $\Sc_{d}$ which are positive semi-definite.\smallskip

\indent $\bullet$ $\Sc_{d}^{++}$ the space of all matrices in $\Sc_{d}$ which are positive definite.\medskip

\noindent We introduce the sets
\[\D=\{\,(x,M)\in \R^{d}\times \Sc_{d}^{+}: M-x\,^{t}\!x \in \Sc_{d}^{+}\,\}.\]
and
\[\D^{\!*}=\{\,(x,M)\in \R^{d}\times \Sc_{d}^{++}: M-x\,^{t}\!x \in \Sc_{d}^{+}\,\}.\]
The two following lemma guarantee that $Z_{n}<+\infty$ pour tout $n\geq 1$.

\begin{lem} If $(x,M)\in \D^{\!*}$ then $\langle M^{-1}x,x\rangle \leq 1$.
\end{lem}


\begin{proof} The matrix $M-x\,^{t}\!x$ is symmetric positive semi-definite. Hence
\[\forall y \in \R^{d}\qquad \langle x,y \rangle^{2}=\langle x\,^{t}\!x\,y,y \rangle \leq \langle My,y\rangle.\]
Applying this inequality to $y=M^{-1}x$, we get
\[\langle x,M^{-1}x \rangle^{2}\leq \langle M^{-1}x,x\rangle.\]
If $x=0$ then $\langle M^{-1}x,x \rangle=0\leq 1$. If $x\neq 0$, since $M\in \Sc^{++}$, we have\linebreak $\langle M^{-1}x,x \rangle > 0$ and thus $\langle M^{-1}x,x \rangle\leq 1$.
\end{proof}

\begin{lem} For any vectors $x_{1},\dots,x_{n}$ in $\R^{d}$ and for any positive real numbers $\l_{1},\dots,\l_{p}$ whose sum is equal to $1$, we have
\[\sum_{i=1}^p\l_i \left(x_i,x_i\,\!^{t}\!x_i\right)\in \D.\]
\label{EnvConv}
\end{lem}

\begin{proof} We denote
\[\nu=\sum_{i=1}^n\l_i\d_{x_i}.\]
This is a probability measure on $\R^{d}$ satisfying
\[m=\int_{\R^{d}} z\,d\nu(z)=\sum_{i=1}^p\l_i x_i\qquad \mbox{and} \qquad S=\int_{\R^{d}}z\,^{t}\!z\,d\nu(z)=\sum_{i=1}^p\l_i\, x_i\,\!^{t}\!x_i\in \Sc_d^{+}.\]
Hence
\[S-m\,\!^{t}\!m=\int_{\R^{d}} (z-m)\,\!^{t}\!(z-m)\,d\nu(z) \in \Sc_d^{+},\]
proving the lemma.
\end{proof}\medskip

\noindent Let $n\geq 1$. These lemma imply that, for any $(x_{1},\dots,x_{n})\in D_n^+$,
\[\left(\frac{1}{n}\sum_{i=1}^{n}x_{i},\frac{1}{n}\sum_{i=1}^{n}x_{i}\,^{t}\!x_{i}\right) \in \D^{\!*}\]
and thus
\[\frac{1}{2} \left\langle \left(\sum_{i=1}^{n}x_{i}\,^{t}\!x_{i}\right)^{-1}\left(\sum_{i=1}^{n}x_{i}\right),\left(\sum_{i=1}^{n}x_{i}\right) \right\rangle\leq \frac{n}{2}.\]
Hence $Z_n\leq e^{n/2}<+\infty$ and the model is well-defined for any $n\geq d$.

\subsection{Large deviations for $(S_n/n,T_n/n)$}
\label{LargeDev}

\noindent As in the one-dimensional case (see~\cite{CerfGorny}), we introduce
\[F: (x,M)\in \D^{\!*} \longmapsto \frac{\langle M^{-1}x,x \rangle}{2}.\]
For any $n\geq d$, the distribution of $(S_{n}/n,T_{n}/n)$ under $\smash{\widetilde{\mu}_{n,\r}}$ is
\[\frac{\displaystyle{\exp(nF(x,M))\ind{\{(x,M) \in \D^{\!*}\}}\,d\widetilde{\nu}_{n,\r}(x,M)}}{\displaystyle{\int_{\D^{\!*}}\exp(nF(s,N))\,d\widetilde{\nu}_{n,\r}(s,N)}},\]
where $\smash{\widetilde{\nu}_{n,\r}}$ is the law of
\[\left(\frac{S_{n}}{n},\frac{T_{n}}{n}\right)=\frac{1}{n}\sum_{i=1}^{n}\left(Y_{i},Y_{i}\,^{t}\!Y_{i}\right)\]
when $Y_{1},\dots,Y_{n}$ are independent random vectors with common law $\r$.\medskip

\noindent We denote by $\langle\,\cdot\,,\,\cdot\,\rangle$ the usual scalar product on $\R^d$ and by $\|\,\cdot\,\|$ the Euclidean norm. We endow $\R^{d}\times \Sc_{d}$ with the scalar product given by
\[((x,M),(y,N))\longmapsto \langle x,y\rangle+\mathrm{tr}(MN)=\sum_{i=1}^{d}x_{i}y_{i}+\sum_{i=1}^{d}\sum_{j=1}^{d}m_{i,j}n_{i,j}.\]
We denote by $\|\,\cdot\,\|_d$ the associated norm.
Notice that
\[\forall z\in \R^{d}\quad\forall A\in \Sc_{d}\qquad\mathrm{tr}(z\,^{t}\!zA)=\sum_{i=1}^{d}\sum_{j=1}^{d}z_{i}z_{j}a_{i,j}=\langle Az,z\rangle.\]
Let $\nu_{\r}$ be the law of $(Z,Z\,^{t}\!Z)$ when $Z$ is a random vector with distribution $\r$.
We define its Log-Laplace $\L$, by
\begin{align*}
\forall (u,A)\in \R^{d}\times \Sc_{d}\qquad\L(u,A)&=\ln \int_{\R^{d}\times \Sc_d}\exp\left(\langle z,y\rangle+\mathrm{tr}(MA)\right)\,d\nu_{\r}(z,M)\\
&=\ln \int_{\R^{d}}\exp\left(\langle u,z\rangle+\langle Az,z\rangle\right)\,d\r(z),
\end{align*}
and its Cram\'er transform $I$ by
\[\forall (x,M) \in \R^d\times \Sc_{d}\qquad I(x,M)=\sup_{(u,A)\in \R^{d}\times\Sc_{d}}\,\big(\,\langle x,u\rangle+\mathrm{tr}(MA)-\L(u,A)\,\big).\]
Let $D_{\L}$ and $D_I$ be the domains of $\R^d\times \Sc_{d}$ where $\L$ and $I$ are respectively finite. All these definitions generalize the case where  $d=1$, treated in~\cite{CerfGorny} and~\cite{Gorny3}.\medskip

\noindent For any $(u,A)\in \R^d\times \Sc_d$, we have
\begin{align*}
\exp\,\L(u,A)&\leq \int_{\R^{d}}\exp\left(\|u\|\,\|z\|+\sqrt{\mathrm{tr}(M^2)}\,\|z\|^2\right)\,d\r(z)\\
&\leq \exp\left(\|(u,M)\|_d\right)+\int_{\R^{d}}\exp\left(\|(u,M)\|_d\,\|z\|^2\right)\,d\r(z).
\end{align*}
Therefore condition $(*)$ is sufficient to ensure that $(0,O_{d})$ belongs to $\Dro_{\L}$, where $O_d$ denotes the $d\times d$ matrix whose coefficients are all zero. As a consequence Cram\'er's theorem (cf.~\cite{DZ}) implies that $(\widetilde{\nu}_{n,\r})_{n\geq 1}$ satisfies the large deviations principle with speed $n$ and governed by $I$.


\section{Convergence in probability of $(S_n/n,T_n/n)$}
\label{ConvProb}

\noindent We saw in the previous section that, under the hypothesis of theorem~\ref{CVloiDimd}, the sequence $(\widetilde{\nu}_{n,\r})_{n\geq 1}$ satisfies the large deviations principle with speed $n$ and governed by $I$. This and Varadhan's lemma (see~\cite{DZ}) suggest that, asymptotically, $(S_n/n,T_n/n)$ concentrates on the minima of the function $I-F$. In subsection~\ref{ConvProb}.\ref{MinI-Fdimd}, we prove that $I-F$ has a unique minimum at $(0,\Sigma)$ on $\D^{\!*}$ and we extend $F$ on the entire closed set $\D$ so that it remains true on $\D$. This is the key ingredient for the proof of the law of large numbers in theorem~\ref{CVloiDimd}, given in subsection~\ref{ConvProb}.\ref{subConvProb}.

\subsection{Minimum de $I-F$}
\label{MinI-Fdimd}

\begin{prop} If $\r$ is a symmetric non-degenerate probability measure on $\R^d$, then
\[\forall x\in \R^d\backslash\{0\}\quad\forall M\in\Sc^{++}_d\qquad I(x,M)>\frac{\langle M^{-1}x,x \rangle}{2}.\]
Moreover, if $\L$ is finite in a neighbourhood of $(0,O_d)$, then the function $I-F$ has a unique minimum at $(0,\Sigma)$ on $\D^{\!*}$.
\label{minI-Fd(1)}
\end{prop}

\begin{proof} Let $x\in \R^d\backslash\{0\}$ and $M\in\Sc^{++}_d$. By taking $A=-M^{-1}x\,^{t}\!xM^{-1}/2$ and $u=M^{-1}x$, we get
\[\langle u,x\rangle+\mathrm{tr}(AM)=\langle M^{-1}x,x\rangle-\frac{1}{2}\mathrm{tr}(M^{-1}x\,^{t}\!x)=\frac{\langle M^{-1}x,x\rangle}{2}.\]
As a consequence
\[I(x,M)\geq \frac{\langle M^{-1}x,x\rangle}{2}-\L\left(M^{-1}x,-\frac{1}{2}M^{-1}x\,^{t}\!xM^{-1}\right).\]
For any $z\in \R^d$, we have $^{t}\!zM^{-1}x=\langle M^{-1}x,z\rangle=\mathrm{tr}(z\,^{t}\!(M^{-1}x)) \in \R$ thus
\[-\frac{1}{2}\mathrm{tr}(z\,^{t}\!zM^{-1}x\,^{t}\!xM^{-1})=-\frac{\langle M^{-1}x,z\rangle}{2}\mathrm{tr}(z\,^{t}\!xM^{-1})=-\frac{\langle M^{-1}x,z\rangle^2}{2}.\]
Therefore
\[\L\left(M^{-1}x,-\frac{1}{2}M^{-1}x\,^{t}\!xM^{-1}\right)=\ln \int_{\R^d}\exp\left(\langle M^{-1}x,z\rangle-\frac{\langle M^{-1}x,z\rangle^2}{2}\right)\,d\r(z).\]
By symmetry of $\r$, we have, for any $s\in \R^d$,
\begin{multline*}
\int_{\R^d}\exp\left(\langle s,z\rangle-\frac{\langle s,z\rangle^2}{2}\right)\,d\r(z)=\int_{\R^d}\exp\left(-\langle s,z\rangle-\frac{\langle s,z\rangle^2}{2}\right)\,d\r(z)\\
=\frac{1}{2}\left(\int_{\R^d}\exp\left(\langle s,z\rangle-\frac{\langle s,z\rangle^2}{2}\right)\,d\r(z)+\int_{\R^d}\exp\left(-\langle s,z\rangle-\frac{\langle s,z\rangle^2}{2}\right)\,d\r(z)\right)\\
=\int_{\R^d}\mathrm{cosh}(\langle s,z\rangle)\,\exp\left(-\frac{\langle s,z\rangle^2}{2}\right)\,d\r(z).
\end{multline*}
As a consequence
\begin{multline*}
\L\left(M^{-1}x,-\frac{1}{2}M^{-1}x\,^{t}\!xM^{-1}\right)=\\
\ln \int_{\R^d}\mathrm{cosh}\left(\langle M^{-1}x,z\rangle\right)\,\exp\left(-\frac{\langle M^{-1}x,z\rangle^2}{2}\right)\,d\r(z).
\end{multline*}
It is straightforward to see that the function $y\longmapsto 1-\mathrm{cosh}(y)\,\exp(-y^2/2)$ is non-negative on $\R$ and vanishes only at~$0$. Hence, for any $z\in \R^d$,
\[\int_{\R^d}\mathrm{cosh}\left(\langle s,z\rangle\right)\,\exp\left(-\frac{\langle s,z\rangle^2}{2}\right)\,d\r(z)\leq 1,\]
and equality holds if and only if $\r(\{z: \langle s,z\rangle=0\})=1$. The non-degeneracy of $\r$ implied that the equality case only holds if $s=0$. Applying this to $s=M^{-1}x\neq 0$, we obtain
\[\L\left(M^{-1}x,-\frac{1}{2}M^{-1}x\,^{t}\!xM^{-1}\right)<0,\]
and thus $I(x,M)>\langle M^{-1}x,x \rangle/2$.\medskip

\noindent Suppose now that $x=0$ and $M\in\Sc^{++}_d$. Then
\[I(x,M)-\frac{\langle M^{-1}x,x \rangle}{2}=I(0,M).\]
If we assume that $\L$ is finite in a neighbourhood of $(0,\dots,0,O_d)$, then $I(0,M)=0$ if and only if $M=\Sigma$ (see proposition~III.4~of~\cite{GornyThesis}). This ends the proof of the lemma.
\end{proof}\medskip

\noindent However, in order to apply Varadhan's lemma, $F$ must be extended to an upper semi-continuous function on the entire closed set $\D$. To this end, we put
\[ \forall (x,M)\in \D\backslash\D^{\!*}\qquad F(x,M)=\frac{1}{2},\]
and it is easy to check that $F$ is indeed an upper semi-continuous function on $\D$.
%
%

\noindent Now we investigate on how the inequality in proposition~\ref{minI-Fd(1)} holds on $\D$.\medskip

\noindent Let $(x,M)\in \R^d\times \Sc_d^+$. We denote by $0\leq \l_1\leq\l_2\leq\dots,\leq \l_d$ the eigenvalues (not necessary distinct) of $M$. There exists an orthogonal matrix $P$ such that $M=PD\,^{t}\!P$, where $D$ is the diagonal matrix such that $D_{i,i}=\l_i$ for any $i\in \{1,\dots,d\}$. We have
\begin{align*}
I(x,M)&=\sup_{(u,A)\in \R^{d}\times\Sc_{d}}\,\big(\,\langle x,u\rangle+\mathrm{tr}(PD\,^{t}\!PA)-\L(u,A)\,\big)\\
&=\sup_{(u,A)\in \R^{d}\times\Sc_{d}}\,\big(\,\langle x,u\rangle+\mathrm{tr}(DA)-\L(u,PA\,^{t}\!P)\,\big).
\end{align*}
Assume that $M\notin \Sc_d^{++}$ and denote by $k=k_M\geq 1$ the dimension of the kernel of $M$. Let $a\in \,]-\infty,0[$. By taking $u=0$ and $A$ the symmetric matrix such that
\[\forall (i,j)\in \{1,\dots,d\}\qquad A_{i,j}=\left\{\begin{array}{ll}
a&\quad \mbox{if}\quad i=j \in \{1,\dots,k\},\\
0&\quad \mbox{otherwise},
\end{array}\right.\]
we obtain
\[I(x,M)\geq -\L(u,PA\,^{t}\!P)=-\ln\int_{\R^d}\exp\,\langle PA\,^{t}\!Pz,z\rangle\,d\r(z),\]
i.e.,
\[\forall a\in \R\qquad I(x,M)\geq-\ln\int_{\R^d}\exp\left(
a\sum_{j=1}^k\big(^{t}\!Pz\big)^2_j\right)\,d\r(z).\]
For any $z\in \R^d$, we have
\begin{align*}
\smash{\sum_{j=1}^k\big(^{t}\!Pz\big)^2_j=0}
&\quad\Longleftrightarrow\quad
\forall j\in \{1,\dots,k\}\quad \langle Pe_j,z\rangle=\langle e_j,^{t}\!Pz\rangle=\big(^{t}\!Pz\big)_j=0\\
&\quad\Longleftrightarrow\quad
\forall x\in \mathrm{Ker}(M)\quad \langle x,z\rangle=0\\
&\quad\Longleftrightarrow\quad
z\in \mathrm{Ker}(M)^{\perp},
\end{align*}
since $(Pe_1,\dots,Pe_k)$ is a basis of $\mathrm{Ker}(M)$ (they are the eigenvectors of $M$ associated to the eigenvalue~$0$). As a consequence
\[\forall z\in \R^d\qquad \exp\left(
a\sum_{j=1}^k\big(^{t}\!Pz\big)^2_j\right)\underset{a\to-\infty}{\longrightarrow}\ind{\mathrm{Ker}(M)^{\perp}}(z).\]
Moreover the left term defines a function which is bounded above by $1$. Therefore the dominated convergence theorem implies that
\[\int_{\R^d}\exp\left(
a\sum_{j=1}^k\big(^{t}\!Pz\big)^2_j\right)\,d\r(z)\underset{a\to-\infty}{\longrightarrow}\r\big(\mathrm{Ker}(M)^{\perp}\big).\]
Whence
\[I(x,M)\geq -\ln\r\big(\mathrm{Ker}(M)^{\perp}\big).\]
So that $I(x,M)>1/2$, it is enough to have $\r\big(\mathrm{Ker}(M)^{\perp}\big)<e^{-1/2}$. Since $\mathrm{Ker}(M)^{\perp}$ is included in some vector hyperplane of $\R^d$, we obtain the following proposition:

\begin{prop} Let $\r$ be a symmetric probability measure on $\R^d$ satisfying $(*)$. Suppose that the $\r$-measure of any vector hyperplane of $\R^d$ is less than $1/\sqrt{e}$. Then $I-F$ has a unique minimum at $(0,\Sigma)$ on $\D$.
\label{minI-Fd(2)}
\end{prop}

\subsection{Convergence of $(S_n/n,T_n/n)$~under $\widetilde{\mu}_{n,\rho}$}
\label{subConvProb}

\noindent Let us first prove the following proposition, which is a consequence of Varadhan's lemma.

\begin{prop} Let $\r$ be a symmetric probability measure on $\R^d$ with a positive definite covariance matrix $\Sigma$. We have
\[\liminfn \frac{1}{n}\ln Z_n\geq 0.\]
Suppose that $\r$ satisfies $(*)$ and that the $\r$-measure of any vector hyperplane of~$\R^d$ is less than $1/\sqrt{e}$. If $\Ac$ is a closed subset of $\R^d\times \Sc_{d}$ which does not contain $(0,\Sigma)$, then
\[\limsupn \frac{1}{n}\ln \int_{\D^{\!*}\cap \Ac}\exp\left(\frac{n\langle M^{-1}x,x \rangle}{2}\right)\,d\widetilde{\nu}_{n,\r}(x,M)<0.\]
\label{TypeVaradhan2}
\end{prop}

\begin{proof} The set ${\Delta}^{\!\!\!\!\raise4pt\hbox{$\scriptstyle o$}}$, the interior of $\D^{\!*}$, contains $(0,\Sigma)$ thus the large deviations principle satisfied by  $(\widetilde{\nu}_{n,\r})_{n\geq 1}$ implies that
\begin{multline*}
\liminfn \frac{1}{n}\ln Z_n=\liminfn \frac{1}{n} \ln \int_{\D^{\!*}}\exp\left(\frac{n\langle M^{-1}x,x \rangle}{2}\right)\,d\widetilde{\nu}_{n,\r}(x,M)\\
\geq \liminfn \frac{1}{n} \ln \widetilde{\nu}_{n,\r}(\D^{\!*})\geq -\inf\,\left\{\,I(x,M):(x,M)\in{\Delta}^{\!\!\!\!\raise4pt\hbox{$\scriptstyle o$}}\,\right\}=0.
\end{multline*}
We prove now the second inequality. Since $\r$ verifies la condition $(*)$, we have that $(0,O_{d})\in \Dro_{\L}$. Cram\'er's theorem (cf.~\cite{DZ}) implies then that $(\widetilde{\nu}_{n,\r})_{n\geq 1}$ satisfies the large deviations principle with speed $n$ and governed by the good rate function~$I$. Since $F$ is upper semi-continuous on the closed set $\D$, Varadhan's lemma (cf.~\cite{DZ}) yields
\begin{multline*}
\limsupn \frac{1}{n}\ln\int_{\D^{\!*}\cap \Ac} \exp\left(\frac{n\langle M^{-1}x,x \rangle}{2}\right)\,d\widetilde{\nu}_{n,\r}(x,M)\\
\leq  \limsupn \frac{1}{n}\ln\int_{\D\cap \Ac} \exp\left(nF(x,M)\right)\,d\widetilde{\nu}_{n,\r}(x,M)\leq \sup_{\D\cap \Ac}\,(\,F-I\,).
\end{multline*}
Since the $\r$-measure of any vector hyperplane of $\R^d$ is less than $1/\sqrt{e}$, proposition~\ref{minI-Fd(2)} implies that $I-F$ has a unique minimum at $(0,\Sigma)$ on $\D$. Since the closed subset $\D\cap \Ac$ does not contain $(0,\Sigma)$ and since $F$ is upper semi-continuous and $I$ is a good rate function, we have
\[\sup_{\D\cap \Ac}\,(\,F-I\,) < 0.\]
This proves the second inequality of the proposition.
\end{proof}
\medskip


\noindent \textbf{Proof of the law of large numbers in theorem~\ref{CVloiDimd}.} Suppose that $\r$ is symmetric, satisfies $(*)$ and that the $\r$-measure of any vector hyperplane of $\R^d$ is less than $1/\sqrt{e}$. Let us denote by $\theta_{n,\r}$ the law of $(S_{n}/n,T_{n}/n)$ under $\smash{\widetilde{\mu}_{n,\r}}$.
Let $U$ be an open neighbourhood of $(0,\Sigma)$ in $\R^{d}\times\Sc_{d}$. Proposition~\ref{TypeVaradhan2} implies that
\begin{multline*}
\limsupn \frac{1}{n}\ln \theta_{n,\r}(U^{c})=\limsupn \frac{1}{n}\ln \int_{\D^{\!*}\cap U^{c}}\exp\left(\frac{n\langle M^{-1}x,x \rangle}{2}\right)\,d\widetilde{\nu}_{n,\r}(x,M)\\
-\liminfn \frac{1}{n}\ln Z_n<0.
\end{multline*}
Hence there exist $\eps >0$ and $n_{0}\geq 1$ such that $\theta_{n,\r}(U^{c}) \leq e^{-n\eps}$ for any $n\geq n_{0}$. Therefore, for any neighbourhood $U$ of $(0,\Sigma)$,
\[\lim_{n \to +\infty}\widetilde{\mu}_{n,\r}\left(\left(\frac{S_{n}}{n},\frac{T_{n}}{n}\right)\in U^{c}\right)=0,\]
i.e., under $\smash{\widetilde{\mu}_{n,\r}}$, $(S_{n}/n,T_{n}/n)$ converges in probability to $(0,\Sigma)$.\qed




\section{Convergence in distribution of $\smash{T_n^{-1/2}\,S_n/n^{1/4}}$\\under $\widetilde{\mu}_{n,\rho}$}
\label{AstuceCasGenDimd}

\noindent In this section, we generalize theorem~1 of~\cite{GorVar} to the higher dimension in order to prove our fluctuation result.

\begin{theo} Let $\r$ be a symmetric non-degenerate probability measure on~$\R^d$ such that
\[\int_{\R^d}\|z\|^5\,d\r(z)<+\infty.\]
Let $\Sigma$ the covariance matrix of $\r$ and let $M_4$ be the function defined in theorem~\ref{CVloiDimd}.
Then, under $\widetilde{\mu}_{n,\r}$,
\[\frac{1}{n^{1/4}}\,T_n^{-1/2}\,S_n \overset{\loi}{\underset{n \to \infty}{\longrightarrow}}
\frac{\displaystyle{\exp\left(-\frac{1}{12}M_4\big(\Sigma^{-1/2}z\big)\right)\,dz}}{\displaystyle{\int_{\R^d}\exp\left(-\frac{1}{12}M_4\big(\Sigma^{-1/2}u\big)\right)\,du}}.\]
\label{TheoGorVarGen}
\end{theo}

\noindent In the proof of this theorem, we show that the limiting law is well defined. Notice that, if $d=1$, then $\Sigma^{-1/2}=\s^{-1}$ and
\[\forall z\in \R\qquad M_4\big(\Sigma^{-1/2}z\big)=\frac{\mu_{4}z^4}{\s^4}.\]
Hence theorem~\ref{TheoGorVarGen} is indeed a generalization of theorem~1 of~\cite{GorVar}

\subsection{Proof of theorem~\ref{TheoGorVarGen}}

\noindent Let $(X_{n}^{k})_{n\geq d,\,1\leq k \leq n}$ be an infinite triangular array of random variables such that, for any $n\geq 1$, $(X_n^1,\dots,X_n^n)$ has the law $\widetilde{\mu}_{n,\r}$. Let us recall that
\[\forall n\geq 1\qquad S_n=X_n^1+\dots+X_n^n\qquad\mbox{and}\qquad T_{n}=X^{1}_{n}\,^{t}\!(X^{1}_{n})+\dots+X^{n}_{n}\,^{t}\!(X^{n}_{n}).\]
and that $T_n \in \Sc_d^{++}$ almost surely. We use the Hubbard-Stratonovich transformation: let $W$ be a random vector with standard multivariate Gaussian distribution and which is independent of $(X_{n}^{k})_{n\geq d,\,1\leq k \leq n}$. Let $n\geq 1$ and $f$ be a bounded continuous function on $\R^d$. We put
\[E_n=\E\left[f\left(\frac{W}{n^{1/4}}+\frac{1}{n^{1/4}}\,T_n^{-1/2}\,S_n\right)\right].\]
We introduce $(Y_i)_{i\geq 1}$ a sequence of independent random vectors with common distribution~$\r$. We denote
\[A_n=\sum_{i=1}^{n}Y_{i},\qquad B_n=\left(\sum_{i=1}^{n}Y_{i}\,^{t}\!Y_{i}\right)^{1/2}\qquad\mbox{and}\qquad \Bc_n=\big\{\mathrm{det}(B_n^2)>0\big\}.\]
We have
\begin{multline*}
E_n=\frac{1}{Z_n(2\pi)^{d/2}}\,\E\Bigg[\ind{\Bc_n}\,\int_{\R^d}f\Bigg(\frac{w}{n^{1/4}}+\frac{1}{n^{1/4}}B_n^{-1}A_n\Bigg)\\
\times\exp\Bigg(\frac{1}{2} \Bigg\langle B_n^{-2}A_n,A_n \Bigg\rangle-\frac{\|w\|^2}{2}\Bigg)\,dw\Bigg].
\end{multline*}
We make the change of variables $z=n^{-1/4}\left(w+B_n^{-1}A_n\right)$ in the integral and we~get
\[E_n=C_n\,\E\Bigg[\ind{\Bc_n }\int_{\R^d}f(z) \exp\Bigg(-\frac{\sqrt{n}\|z\|^2}{2}+n^{1/4}\big\langle z, B_n^{-1}A_n\big\rangle\Bigg)\,dz\Bigg]\]
where $C_n=n^{d/4}Z_n^{-1}(2\pi)^{-d/2}$.
Let $U_1,\dots,U_n,\eps_1,\dots,\eps_n$ be independent random variables such that the distribution of $U_i$ is $\r$ and the distribution of $\eps_i$ is $(\d_{-1}+\d_1)/2$, for any $i\in \{1,\dots,n\}$. Since $\r$ is symmetric, the random variables $\eps_1U_1,\dots,\eps_nU_n$ are also independent with common distribution~$\r$. Therefore
\[E_n=C_n\,\E\Bigg[\ind{ \Bc_n}\int_{\R^d}f(z) \exp\Bigg(-\frac{\sqrt{n}\|z\|^2}{2}+n^{1/4}\Bigg\langle z, B_n^{-1}\Bigg(\sum_{i=1}^{n}\eps_iU_{i}\Bigg)\Bigg\rangle\Bigg)\,dz\Bigg].\]
In the case where the matrix $B_n^2=U_1\,\!^{t}\!U_1+\dots+U_n\,\!^{t}\!U_n$ is invertible, we denote
\[\forall i\in\{1,\dots,n\}\qquad a_{i,n}=\Bigg(\sum_{j=1}^nU_j\,\!^{t}\!U_j\Bigg)^{-1/2}U_i.\]
By using Fubini's lemma and the independence of $\eps_i,U_i$, $i\geq 1$, we obtain
\begin{multline*}
E_n=C_n\,\E\Bigg[\ind{\Bc_n}\,\int_{\R^d}f(z)\exp\left(-\frac{\sqrt{n}\|z\|^2}{2}\right)\\
\times\E\left(\,\prod_{i=1}^n\exp\left(n^{1/4}\eps_i \langle z,a_{i,n}\rangle\right)\,\Bigg|\,(U_1,\dots,U_n)\,\right)\,dz\Bigg].
\end{multline*}
Therefore
\[E_n=C_n\E\Bigg[\ind{\Bc_n}\int_{\R^d}\!f(z)\exp\left(-\frac{\sqrt{n}\|z\|^2}{2}\right)\exp\left(\sum_{i=1}^n\ln\mathrm{cosh}\,(n^{1/4} \langle z,a_{i,n}\rangle)\right)\,dz\Bigg].\]
We define the function $g$ by
\[\forall y\in \R\qquad g(y)=\ln\mathrm{cosh}\,y-\frac{y^2}{2}.\]
It is easy to see that $g(y)<0$ if $y>0$. We notice that, for any $x$ and $y$ in $\R^d$, $\langle x,y\rangle^2=\langle x, (y \,^t\!y) x\rangle$. Therefore
\[\sum_{i=1}^n\langle z,a_{i,n}\rangle^2=\sum_{i=1}^n\langle z, (a_{i,n} \,^t\!a_{i,n}) z\rangle=\left\langle z,\left(\sum_{i=1}^n a_{i,n} \,^t\!a_{i,n}\right) z\right\rangle=\left\langle z,\mathrm{I}_d z\right\rangle=\|z\|^2.\]
As a consequence
\[E_n=C_n\,\E\left[\ind{\Bc_n}\,\int_{\R^d}f(z)\exp\left(\sum_{i=1}^n g(n^{1/4} \langle z,a_{i,n}\rangle)\right)\,dz\right].\]
Now we use Laplace's method. Let us examine the convergence of the term in the exponential: for any $z\in \R^d$ and $i\in \{1,\dots,n\}$, the Taylor-Lagrange formula states that there exists a random variable $\xi_{n,i}$ such that
\[g\big(n^{1/4}\langle z,a_{i,n}\rangle\big)=-\frac{n\langle z,a_{i,n}\rangle^4}{12}+\frac{n^{3/2}\langle z,a_{i,n}\rangle^5}{n^{1/4}5!}g^{(5)}(\xi_{n,i}).\]
Let $z\in \R^d$. We have
\[n\sum_{i=1}^n \langle z,a_{i,n}\rangle^4=n\sum_{i=1}^n \left\langle B_n^{-1} z,U_i\right\rangle^4=\frac{1}{n}\sum_{i=1}^n \left\langle \sqrt{n} B_n^{-1}z,U_i\right\rangle^4.\]
We denote $\zeta_n=\sqrt{n} B_n^{-1}z$. We have
\begin{multline*}
n\sum_{i=1}^n \langle z,a_{i,n}\rangle^4=\frac{1}{n}\sum_{i=1}^n \left\langle \z_n,U_i\right\rangle^4=\frac{1}{n}\sum_{i=1}^n \left( \sum_{j=1}^d(\z_n)_{j}(U_i)_j\right)^4\\
=\sum_{1\leq j_1,j_2,j_3,j_4\leq d}(\z_n)_{j_1}(\z_n)_{j_2}(\z_n)_{j_3}(\z_n)_{j_4}\,\frac{1}{n}\sum_{i=1}^n(U_i)_{j_1}(U_i)_{j_2}(U_i)_{j_3}(U_i)_{j_4}.
\end{multline*}
Since $\r$ is non-degenerate, its covariance matrix $\Sigma$ is invertible. Moreover $\r$ has a finite fourth moment thus the law of large number implies that
\[\z_n\overset{\mbox{a.s}}{\underset{n\to+\infty}{\longrightarrow}} \Sigma^{-1/2}z,\]
and that, for any $(j_1,j_2,j_3,j_4)\in \{1,\dots,d\}^4$,
\[\frac{1}{n}\sum_{i=1}^n(U_i)_{j_1}(U_i)_{j_2}(U_i)_{j_3}(U_i)_{j_4}\overset{\mbox{a.s}}{\underset{n\to+\infty}{\longrightarrow}} \int_{\R^d}y_{j_1}y_{j_2}y_{j_3}y_{j_4}\,d\r(y).\]
As a consequence
\[n\sum_{i=1}^n \langle z,a_{i,n}\rangle^4\overset{\mbox{a.s}}{\underset{n\to+\infty}{\longrightarrow}} M_4\left(\Sigma^{-1/2}z\right).\]
Since $\r$ has a finite fifth moment, we prove similarly that
\[n^{3/2} \sum_{i=1}^n \langle z,a_{i,n}\rangle^5\overset{\mbox{a.s}}{\underset{n\to+\infty}{\longrightarrow}} M_5\left(\Sigma^{-1/2}z\right),\]
where, for any $z\in \R^d$,
\[M_5(z)=\sum_{1\leq j_1,j_2,j_3,j_4,j_5\leq d}\left(\int_{\R^d}y_{j_1}y_{j_2}y_{j_3}y_{j_4}y_{j_5}\,d\r(y)\right)z_{j_1}z_{j_2}z_{j_3}z_{j_4}z_{j_5}.\]
Finally, by a simple computation, we see that $g^{(5)}$ is bounded over $\R$. Hence
\[\forall z\in \R^d\qquad \sum_{i=1}^n g\big(n^{1/4}\langle z,a_{i,n}\rangle\big)\overset{\mbox{a.s}}{\underset{n\to+\infty}{\longrightarrow}}-\frac{1}{12}M_4\left(\Sigma^{-1/2}z\right).\]

\begin{lem} There exists $c>0$ such that
\[\forall z\in \R^d\quad \forall n\geq 1\qquad \sum_{i=1}^n g\big(n^{1/4}\langle z,a_{i,n}\rangle\big)\leq -\frac{c\|z\|^4}{1+\|z\|^2/\sqrt{n}}.\]
\end{lem}

\noindent The proof of this lemma follows the same lines than the proof of lemma~3~in~\cite{GorVar}.\medskip


\noindent If $\|z\|\leq n^{1/4}$ then $1+\|z\|^2/\sqrt{n}\leq 2$ and thus, by the previous lemma,
\[\left|\ind{\Bc_n}\,\ind{\|z\|\leq n^{1/4}}\,\exp\left(\sum_{i=1}^n g\big(n^{1/4}\langle z,a_{i,n}\rangle\big)\right)\right|\leq \exp\left(-\frac{c\|z\|^4}{2}\right).\]
Thus the dominated convergence theorem implies that
\[z\longmapsto \exp\left(-M_4\left(\Sigma^{-1/2}z\right)/12\right)\]
is integrable on $\R^d$ and that
\begin{multline*}
\E\left[\ind{\Bc_n}\,\int_{\R^d}\ind{|z|\leq n^{1/4}}\,f\left(z\right)\exp\left(\sum_{i=1}^n g\big(n^{1/4}\langle z,a_{i,n}\rangle\big)\right)\,dz\right]\\
\underset{n\to+\infty}{\longrightarrow} \int_{\R^d}f(z)\exp\left(-\frac{1}{12}M_4\left(\Sigma^{-1/2}z\right)\right)\,dz.
\end{multline*}

\noindent If $\|z\|> n^{1/4}$ then $1+\|z\|^2/\sqrt{n}\leq 2\|z\|^2/\sqrt{n}$ and thus, by the previous lemma,
\begin{multline*}
\E\left[\ind{\Bc_n}\,\int_{\R^d}\ind{|z|> n^{1/4}}\,f\left(z\right)\exp\left(\sum_{i=1}^n g\big(n^{1/4}\langle z,a_{i,n}\rangle\big)\right)\,dz\right]\\
\leq \left\|f\right\|_{\infty}\int_{\R^d}\exp\left(-\frac{c\sqrt{n}\|z\|^2}{2}\right)\,dz=\frac{\left\|f\right\|_{\infty}(2\pi)^{d/2}}{n^{d/4}c^{d/2}}\underset{n\to+\infty}{\longrightarrow} 0,
\end{multline*}
and thus
\begin{multline*}
\frac{E_n}{C_n}=\E\left[\ind{\Bc_n}\,\int_{\R^d}f\left(z\right)\exp\left(\sum_{i=1}^n g\big(n^{1/4}\langle z,a_{i,n}\rangle\big)\right)\,dz\right]\\
\underset{n\to+\infty}{\longrightarrow} \int_{\R^d}f(z)\exp\left(-\frac{1}{12}M_4\left(\Sigma^{-1/2}z\right)\right)\,dz.
\end{multline*}
If we take $f=1$, we get
\[\frac{1}{C_n}=\frac{Z_n(2\pi)^{d/2}}{n^{d/4}}\underset{n\to+\infty}{\longrightarrow}\int_{\R^d}\exp\left(-\frac{1}{12}M_4\left(\Sigma^{-1/2}z\right)\right)\,dz.\]
By Paul Levy's theorem, we have then
\[\frac{W}{n^{1/4}}+\frac{1}{n^{1/4}}\,T_n^{-1/2}\,S_n \overset{\loi}{\underset{n \to \infty}{\longrightarrow}}
\frac{\displaystyle{\exp\left(-\frac{1}{12}M_4\big(\Sigma^{-1/2}z\big)\right)\,dz}}{\displaystyle{\int_{\R^d}\exp\left(-\frac{1}{12}M_4\big(\Sigma^{-1/2}u\big)\right)\,du}}.\]
Since $(Wn^{-1/4})_{n\geq 1}$ converges in distribution to $0$, Slutsky lemma (theorem~3.9 of~\cite{BillCVPM}) implies the convergence in distribution of theorem~\ref{TheoGorVarGen}.\medskip

\noindent We remark that the hypothesis that $\r$ has a finite fifth moment may certainly be weakened by assuming instead that
\[\exists \eps>0\qquad \int_{\R^d}\|z\|^{4+\eps}\,d\r(z)<+\infty.\]

\subsection{Proof of the fluctuation result in theorem~\ref{CVloiDimd}.}
\label{SecCVloiDimd}

In section~\ref{ConvProb}, we proved the law of large numbers in theorem~\ref{CVloiDimd}. It implies that, under $\widetilde{\mu}_{n,\r}$, $T_{n}/n$ converges in probability to $\Sigma$. Moreover the hypothesis $(*)$ implies that $(0,O_d)\in \Dro_{\L}$ and thus $\r$ has finite moments of all orders. Theorem~\ref{TheoGorVarGen} and Slutsky lemma yield 
\[\frac{S_n}{n^{3/4}}=\left(\frac{T_n}{n}\right)^{1/2}\times\frac{1}{n^{1/4}}\,T_n^{-1/2}\,S_n\overset{\loi}{\underset{n \to \infty}{\longrightarrow}} \frac{\displaystyle{\exp\left(-\frac{1}{12}M_4\big(\Sigma^{-1}z\big)\right)\,dz}}{\displaystyle{\int_{\R^d}\exp\left(-\frac{1}{12}M_4\big(\Sigma^{-1}u\big)\right)\,du}}.\]
Theorem~\ref{CVloiDimd} is proved.

\bibliographystyle{plain}
\bibliography{biblio}\bigskip

\begin{flushleft}
\footnotesize \textbf{Matthias Gorny}\\
Laboratoire de Math\'ematiques d'Orsay,\\
Universit\'e Paris-Sud, CNRS, Universit\'e Paris-Saclay,\\
91405 Orsay, France\\
\texttt{matthias.gorny@math.u-psud.fr}
\end{flushleft}

\end{document}